\documentclass[11pt]{article}
\usepackage{amsmath,amssymb,amsthm}
\usepackage{mathtools}
\usepackage{mathdots}
\usepackage{hyperref,url}
\hypersetup{
	colorlinks=true,
  linkcolor=black,          
  citecolor=black,         
  filecolor=black,      
  urlcolor=black           
}

\newcommand{\tr}{\operatorname{tr}}

\newcommand{\PERM}{\textnormal{perm}}

\DeclareMathOperator{\diag}{diag}
\newcommand{\NN}{\mathbb{N}}
\newcommand{\ZZ}{\mathbb{Z}}
\newcommand{\QQ}{\mathbb{Q}}
\newcommand{\RR}{\mathbb{R}}

\newcommand{\seqnum}[1]{\href{https://oeis.org/#1}{\rm \underline{#1}}}

\title{Proofs of several OEIS conjectures on determinants and permanents}
\author{ Sela Fried\\ Department of Computer Science, Israel Academic College\\ Ramat Gan, Israel\\ \texttt{friedsela@gmail.com} }

\date{}

\theoremstyle{plain}
\newtheorem{theorem}{Theorem}
\newtheorem{lemma}{Lemma}
\newtheorem{corollary}{Corollary}
\newtheorem{example}{Example}
\newtheorem{remark}{Remark}

\begin{document}
\maketitle

\begin{abstract} 
We prove several conjectures recorded in the On-Line Encyclopedia of Integer Sequences. The conjectures considered here concern determinants and permanents of special matrices, such as Toeplitz matrices, cross matrices, Kronecker powers of matrices, and matrices whose entries are defined by powers of differences. As tools we use row and column operations, block determinant formulas, Cauchy determinants, Sylvester's determinant theorem, and $LU$-factorizations. We also obtain closed-form formulas for several related integer sequences for which no such formulas were conjectured. 
\end{abstract} 


\medskip 


\section{Introduction} 
The On-Line Encyclopedia of Integer Sequences (OEIS) \cite{Sl} is a rich source of experimentally discovered patterns and conjectures. In previous works we settled several conjectures from the OEIS concerned with combinatorics, number theory, and symbolic enumeration \cite{FriedSeveralOEIS}, \cite{FriedTenOEIS}. Here we continue this line of work, focusing on conjectures involving determinants and permanents of structured matrices. Structured matrices such as Toeplitz matrices, Cauchy matrices, tridiagonal matrices, and matrices with prescribed sparsity patterns play a central role throughout enumerative combinatorics and linear algebra. Many sequences in the OEIS correspond to the determinants or permanents of such matrices and are sometimes accompanied by conjectured closed-form formulas. The goal of this work is to give rigorous proofs for some of these formulas. Occasionally, we provide formulas for sequences for which no formula was conjectured.

We begin with a correction of a conjecture concerned with the permanent of a matrix, defined using the floor function (\seqnum{A000051}). We then prove a determinant formula related to the Dowling numbers (\seqnum{A007405}). The next sections are concerned with the determinants of special matrices: an almost cross matrix (\seqnum{A071999}); a symmetric Toeplitz matrix (\seqnum{A083392}); a symmetric matrix with entries $|i^2-j^2|$ (\seqnum{A085799}); a Toeplitz matrix (\seqnum{A323254}); a symmetric matrix defined using the $\min$ function (\seqnum{A351154}); a Hermitian Toeplitz matrix (\seqnum{A359559}). We then prove conjectures of Sun related to a family of Toeplitz-type determinants in \seqnum{A355175} and \seqnum{A355326}. Finally, we prove a conjecture related to binary matrices (\seqnum{A250742}), and a formula for the determinant of principal submatrices of Kronecker powers (\seqnum{A094384}). 

All proofs are self-contained. In the next section we fix the notation and recall well-known results that we use.

\section{Preliminaries}

We let $\NN$ denote the set of natural numbers $\{1,2,\ldots\}$ and set $\NN_0=\NN\cup\{0\}$. For $m\in \NN_0$ we set $[m]=\{1,2,\ldots,m\}$, with $[0]=\emptyset$. For a set $X$ we write $\# X$ for the number of elements of $X$. For $i,j\in\NN_0$, the Kronecker delta $\delta_{ij}$ is equal to $1$ if $i=j$ and $0$ otherwise. If $A$ is a matrix, we write $A(i,j)$ for the entry of $A$ in the $i$th row and the $j$th column. We denote by $I_n$ the identity matrix of size $n$. We write $A^T$ for the transpose of $A$. For a square matrix $A$ we write $\tr(A)$ for the trace of $A$. For $x \in\RR^n$ we write $\diag(x)$ for the diagonal matrix of size $n$ with the entries of $x$ on the main diagonal. For $n\in\NN$ we write $S_n$ for the symmetric group of degree $n$. If $A$ is a square matrix of size $n$, we write $\PERM(A)$ for the permanent of $A$, i.e.,
\begin{equation}\label{zp1}
\PERM(A)
= \sum_{\sigma\in S_n} \prod_{j=1}^n A(j,\sigma(j)).
\end{equation}

\begin{lemma}\label{l1}

\begin{enumerate}
\item Let $x\in\RR$. Then
\begin{equation}\label{pp1}
\lfloor 1-x \rfloor = 1-\lceil x \rceil.
\end{equation}
\item Let $n\in\NN$ and let $k\in\ZZ$. Then
\begin{equation}\label{pp11}
\left\lceil \frac{k}{n} \right\rceil - 1 =\left\lfloor \frac{k-1}{n} \right\rfloor.
\end{equation}
\end{enumerate}
\end{lemma}

\begin{proof}
\begin{enumerate}
\item Let $n=\lceil x\rceil$. Then
$n-1<x\le n$.
Thus, $-n\le -x<1-n$ and therefore
$1-n\le 1-x<2-n$. Hence
$\lfloor 1-x\rfloor=1-n=1-\lceil x\rceil$.

\item Set $t = \lceil k/n \rceil$.
Then $t-1 < k/n \leq t$ and therefore $n(t-1) < k \leq nt$. Since both $k$ and $n(t-1)$ are integers, necessarily
$n(t-1)+1\leq k$.
Thus $t-1\leq (k-1)/n$.
On the other hand, from $k\leq nt$ it follows
$(k-1)/n \leq (nt-1)/n< t$. It follows that $t-1 \leq  (k-1)/n < t$, meaning that 
\[
\left\lfloor \frac{k-1}{n} \right\rfloor = t-1
= \left\lceil \frac{k}{n} \right\rceil - 1,
\]
as asserted.\qedhere
\end{enumerate}
\end{proof}

A cross matrix is a matrix in which the elements outside the main and secondary diagonals are all zero. The following result, which gives the determinant of a cross matrix, is Theorem 3.1 in \cite{LiuCross}.
\begin{lemma}\label{l171999}
Let $n\in\NN$ and consider a cross matrix $X$ of size $n$ 
\[
X=\begin{pmatrix}
x_{11} & 0      & 0 & \cdots & 0 & x_{1n}\\
0  & x_{22} & 0  & \cdots & x_{2,n-1} & 0\\
0  & 0 & x_{33} & \iddots & 0          & 0\\
\vdots & \vdots & & \ddots & \vdots       & \vdots\\
0      & x_{n-1,2} & 0   & \cdots & x_{n-1,n-1}  & 0\\
x_{n1} & 0      & 0      & \cdots & 0            & x_{nn}
\end{pmatrix}.
\] Write $n=2k+\varepsilon$, where $k\in\NN_0$ and $\varepsilon\in\{0,1\}$. Then
\[
\det(X)=
\begin{cases}
\displaystyle \prod_{i=1}^{k}(x_{ii}x_{n+1-i,n+1-i}-x_{i,n+1-i}x_{n+1-i,i}),
& \textnormal{if $\varepsilon=0$},\\
\displaystyle x_{k+1,k+1}\prod_{i=1}^{k}(x_{ii}x_{n+1-i,n+1-i}-x_{i,n+1-i}x_{n+1-i,i}),
& \textnormal{if $\varepsilon=1$}.
\end{cases}
\] In particular, if $x_{k+1,k+1}=1$, then
\begin{equation}\label{e171999}
\det(X)=\prod_{i=1}^{\lfloor\frac{n}{2}\rfloor}(x_{ii}x_{n+1-i,n+1-i}-x_{i,n+1-i}x_{n+1-i,i}).
\end{equation}
\end{lemma}

The following result is known as the Schur determinant formula (e.g., \cite[(3.4.9) and P.3.7]{GarciaHorn}).

\begin{lemma}\label{Schur}
Let $M=\begin{pmatrix} A& B\\ C &D \end{pmatrix}$ be a block matrix.     
\begin{enumerate}
\item Assume that $A$ is invertible. Then $\det(M) = \det(A)\det(D-CA^{-1}B)$.
\item Assume that $D$ is invertible. Then $\det(M) = \det(D)\det(A-BD^{-1}C)$.
\end{enumerate}
\end{lemma}

The following result is folklore (e.g., \cite{stack}) and is a special case of \cite[(1.2)]{Usmani}.

\begin{lemma}\label{tri}
Let $a,b\in\RR$ and $n\in\NN$. Consider the symmetric tridiagonal Toeplitz matrix 
\[
T_n=
\begin{pmatrix}
b & a & 0 & \cdots & 0\\
a & b & a & \ddots & \vdots\\
0 & a & b & \ddots & 0\\
\vdots & \ddots & \ddots & \ddots & a\\
0 & \cdots & 0 & a & b
\end{pmatrix}.
\] 
\begin{enumerate}
\item Assume that $b^2\neq 4a^2$. Let $\lambda_{\pm}=\frac{b\pm\sqrt{b^2-4a^2}}{2}$. Then
\[
\det(T_n)=\frac{\lambda_+^{\,n+1}-\lambda_-^{\,n+1}}{\lambda_+-\lambda_-}.
\]
\item Assume that $b^2= 4a^2$. Then
\[
\det(T_n)=(n+1)\left(\frac{b}{2}\right)^n.
\]
\end{enumerate}
\end{lemma}

The following result is known by the name Sylvester's determinant theorem (e.g., \cite[(B.1.16)]{PP}).

\begin{lemma}
Let $m,n\in\NN$ and let $U$ and $V$ be two $n\times m$ matrices.
Then 
\begin{equation}\label{SYL}
\det(I_n+UV^T) = \det(I_m+V^TU).
\end{equation}
\end{lemma}

The following result is known by the name Jacobi's formula (e.g., \cite[(46)]{MCB}).

\begin{lemma}\label{Jac}
Let $A=A(t)$ be a square matrix parameterized by $t$. Then 
\[
\frac{d}{dt}\det(A)=\det(A)\tr\left(A^{-1}\frac{d}{dt}A\right).
\]
\end{lemma}

Let $n\in\NN$. A Cauchy matrix is a square matrix $A$ of size $n$ such that there exist real numbers $a_1,\ldots,a_n,b_1,\ldots,b_n$ with
$a_i+b_j\neq 0$, for every $i,j\in[n]$, and such that 
\[A(i,j)=\frac{1}{a_i+b_j}, \qquad i,j\in[n].\]

The following result may be found, for example in  \cite[0.9.12]{HH}.
\begin{lemma}\label{Cauchy}
Let $n\in\NN$ and let $A$ be a Cauchy matrix of size $n$. Then
\[
\det(A)
=
\frac{
\displaystyle\prod_{1\leq i<j\leq n}(a_j-a_i)(b_j-b_i)
}{
\displaystyle\prod_{1\leq i,j\leq n}(a_i+b_j)
}.
\]
\end{lemma}

\section{Main results}

\subsection{\texorpdfstring{\seqnum{A000051}}{}}

For $n\in\NN_0$ consider the square matrix $M_n$ of size $n+1$ given by
\begin{equation}\label{e1}
M_n(j,k) = -\left\lfloor\frac{j-k-1}{n+1}\right\rfloor,
\qquad j,k\in[n+1].
\end{equation} In a comment to \seqnum{A000051} Luschny conjectured that $\PERM(M_n)=2^n+1$. We show that the conjecture, as stated, is false. However, a minor adjustment to the definition of \(M_n\) yields a valid statement, which was likely Luschny's original intention.

\begin{lemma}
Let $n\in\NN_0$ and let $M_n$ be the matrix defined by \eqref{e1}. Then 
\begin{equation}\label{e2}
M_n(j,k)=\begin{cases}
1, & \textnormal{if } k\geq j,\\
0, & \textnormal{if } k<j,
\end{cases}
\qquad j,k\in[n+1].
\end{equation} In particular, 
$\PERM(M_n) = 1$.
\end{lemma}

\begin{proof}
Let $j,k\in [n+1]$. We claim that
\[
\left\lfloor\frac{j-k-1}{n+1}\right\rfloor =
\begin{cases}
-1, & \textnormal{if } k\geq j,\\
0, & \textnormal{if } k<j.
\end{cases}
\] Indeed, suppose first that $k\geq j$. Then $-(n+1)\leq j-k-1 < 0$. Thus,
\[
-1 \le \frac{j-k-1}{n+1}<0\Longrightarrow\left\lfloor\frac{j-k-1}{n+1}\right\rfloor=-1.
\] Assume now that $k<j$. Then
\[
0 \le j-k-1 \le n-1,
\] Thus,
\[
0 \leq \frac{j-k-1}{n+1} \leq \frac{n-1}{n+1}<1\Longrightarrow\left\lfloor\frac{j-k-1}{n+1}\right\rfloor=0.
\]
This concludes the proof of \eqref{e2}. Thus, $M_n$ is an upper triangular matrix. In this case the permanent is the product of the diagonal entries. Thus, $\PERM(M_n)=1$.
\end{proof}

\begin{theorem}
Let $n\in\NN_0$ and consider the square matrix $M_n$ of size $n+1$ given by
\begin{equation}\label{e1q}
M_n(j,k) = -\left\lfloor\frac{j-k-2}{n+1}\right\rfloor,
\qquad j,k\in[n+1].
\end{equation}
Then $\PERM(M_n) = 2^n+1$.
\end{theorem}

\begin{proof}
For $n=0$, the matrix is $(2)$ and its permanent is $2=2^0+1$, as asserted. Thus, assume that $n\geq1$. First, we claim that we may replace the definition of $M_n(j,k)$ in \eqref{e1q} with 
\begin{equation}\label{e2q}
M_n(j,k) = \left\lfloor\frac{j+k}{n+1}\right\rfloor.
\end{equation} Indeed, the substitution $j\mapsto n+2-j$ reverses the order of the rows of $M_n$ and for the resulting matrix we have
\[
-\left\lfloor \frac{(n+2-j)-k-2}{n+1}\right\rfloor =-\left\lfloor 1-\frac{j+k+1}{n+1}\right\rfloor.\] Using \eqref{pp1} and \eqref{pp11} we obtain
\[
-\left\lfloor 1-\frac{j+k+1}{n+1}\right\rfloor = \left\lceil \frac{j+k+1}{n+1} \right\rceil - 1=\left\lfloor \frac{j+k}{n+1} \right\rfloor.
\]
Since the permanent of a matrix is invariant under permutations of its rows, we may calculate $\PERM(M_n)$ using the alternative definition of $M_n$ given by \eqref{e2q}. To this end, set $N=n+1$. Clearly,
\[
M_n(j,k) =
\begin{cases}
0,& \textnormal{if } 2\leq j+k \leq N-1,\\
1,& \textnormal{if } N \leq j+k \leq 2N-1,\\
2,& \textnormal{if } j+k=2N.
\end{cases}
\] Thus, a term in the permanent expansion for $M_n$ corresponding to $\sigma\in S_N$ contributes a nonzero value if and only if
$\sigma\in S(N)$, where
\[
S(N) = \{\sigma\in S_N\; :\; j+\sigma(j)\geq N\textnormal{ for every } j\in[N]\}.
\] Let $\sigma\in S(N)$. Then 
\[
\prod_{j=1}^{N}M_{n}(j,\sigma(j))=\begin{cases}
2, & \textnormal{if }\sigma(N)=N,\\
1, & \textnormal{otherwise}.
\end{cases}
\]
Let
$C(N) = \{\sigma\in S(N)\;:\; \sigma(N)=N\}$. Then,
$\PERM(M_n) = \#S(N) + \# C(N)$. We prove that $\#S(N) = 2^{N-1}$ and $\# C(N) = 1$.
To this end, define a map $\Phi\colon S(N)\to S_N$ as follows: for $\sigma\in S(N)$ and $1\leq j\leq N$, set $\pi(j) = N+1-\sigma(j)$. Clearly, $\pi\in S_N$ and, for every $j\in[N]$, we have
\[
j+\sigma(j)\geq N
\iff \pi(j) \leq j+1.\] Thus, $\Phi$ is a bijection from $S(N)$ to the set
\[
S'(N) = \{\pi\in S_N \;:\; \pi(j)\le j+1\textnormal{ for every } j\in[N]\}.
\]
Thus, $\#S(N) = \#S'(N)$. Let $C'(N)=\{\pi\in S'(N)\;:\;\pi(N)=1\}$.
Since $\sigma(N)=N$ if and only if $\Phi(\sigma)(N)=1$, we have
$\# C(N) = \#C'(N)$. It remains to establish $\# S'(N)$ and $\# C'(N)$. To this end, let $\pi\in S'(N)$. For $\pi(1)$ there are two possibilities, namely $1$ and $2$. Then the possible values for $\pi(2)$ are in $[3]\setminus\{\pi(1)\}$, i.e., two possibilities. Similarly, if $2\leq j\leq N-1$, the possible values for $\pi(j)$ are in $[j+1]\setminus\{\pi(1),\ldots,\pi(j-1)\}$. Again, two possibilities. Finally, the possible values for $\pi(N)$ are in $[N]\setminus\{\pi(1),\ldots,\pi(N-1)\}$, yielding only one possibility. This proves that $\# S'(N)=2^{N-1}$. Now let $\pi\in C'(N)$. We prove by induction that $\pi(j)=j+1$, for every $j\in[N-1]$. By definition of $S'(N)$, we have $\pi(1)\in[2]$. Since $\pi(N)=1$, necessarily $\pi(1)=2$. Assume that the assertion holds for every $j\in[k]$, for some $k\in[N-2]$. The possible values for $\pi(k+1)$ are in \[[k+2]\setminus\{\pi(1),\ldots,\pi(k),\pi(N)\}=[k+2]\setminus\{2,3,\ldots,k+1,1\}.\] Thus, $\pi(k+1)=k+2$, concluding the inductive proof. It follows that $\#C'(N)=1$ and the proof is complete.
\end{proof}

\subsection{\texorpdfstring{\seqnum{A007405}}{}}

The Dowling numbers are a combinatorial sequence, denoted by $(D_n)_{n\in\NN_0}$, enumerating objects such as $B$-type set partitions \cite[Sec.~4.1]{Adler} and  flattened Stirling permutations \cite[Corollary 3.11]{Buck}. By \cite[Theorem 7]{Ben}, the exponential generating function of the Dowling numbers is given by 
\begin{equation}\label{doweq}
\sum_{n=0}^\infty D_n\frac{x^n}{n!}=\exp\left(x+\frac{e^{2x}-1}{2}\right).\end{equation} The statement of Theorem \ref{Dow1} was conjectured by Irwin in \seqnum{A007405}. We shall need the following result.

\begin{lemma}\label{l17405}
For $n\in\NN$ let $B_n$ be the square matrix of size $n$ defined by
\begin{equation}\label{g1}
B_n(i,j)= \begin{cases} 
0,& \textnormal{if } i<j-1,\\ 
2,& \textnormal{if } i=j-1,\\ 
-2,& \textnormal{if } i=j,\\ 
-\binom{n-j}{i-j},& \textnormal{if } i>j, \end{cases} 
\qquad i,j\in[n].
\end{equation} Set $b_0=1$ and $b_n=\det(B_n)$. Then, for every $n\in \NN_0$,
\begin{equation}\label{ev17405}
b_{n+1}=-b_n-\sum_{k=0}^{n}(-1)^k\binom{n}{k}2^kb_{n-k}.
\end{equation}
\end{lemma}

\begin{proof}
Let $n\in\NN_0$. We expand $\det(B_{n+1})$ along the first column. For $k\in[n+1]$
let $M_k$ denote the minor obtained from $B_{n+1}$ by deleting the first column and the $k$th row. It is straightforward to verify that $M_1=B_n$. Furthermore, $B_{n+1}(1,1)=-2$ and, for $2\leq k\leq n+1$, we have $B_{n+1}(k,1)=-\binom{n}{k-1}$. Thus, 
\begin{equation}\label{em17405}
b_{n+1}=-2b_n - \sum_{k=1}^{n}(-1)^k\binom{n}{k}\det(M_{k+1}).    
\end{equation}
Let $k\in[n]$ and consider $M_{k+1}$. We claim that $M_{k+1}$ has the following representation as a block matrix: 
\[
M_{k+1}=
\begin{pmatrix}
P_k & 0\\
\ast & Q_{n-k}
\end{pmatrix},
\] where $P_k$ and $Q_{n-k}$ are square matrices of sizes $k$ and $n-k$, respectively. Indeed, let $i\in[k]$ and $k+1\leq j\leq n$. Then $i\leq k < k+1\leq j + 1 - 1$. Thus, $M_{k+1}(i,j) = B_{n+1}(i,j+1)=0$. It follows that 
\begin{equation}\label{eba17405}
\det(M_{k+1})=\det(P_k)\det(Q_{n-k}).
\end{equation} We now claim that $\det(P_k)=2^k$. To see this, let $i,j\in [k]$. We have $P_k(i,j)=B_{n+1}(i,j+1)$ and the latter is equal to $0$ if $i<j+1-1=j$ and to $2$ if $i=j+1-1=j$. It follows that $P_k$ is a lower-triangular matrix with all diagonal entries equal to $2$. Thus, $\det(P_k)=2^k$. We now claim that $Q_{n-k}=B_{n-k}$ and therefore $\det(Q_{n-k})=\det(B_{n-k})$. Indeed, let $i,j\in [n-k]$. We have $Q_{n-k}(i,j)=B_{n+1}(i+k+1,j+k+1)$. Now, each of the conditions in \eqref{g1} holds for $i$ and $j$ if and only if it holds for $i+k+1$ and $j+k+1$. Regarding the fourth condition, notice that $-\binom{n+1-(j+k+1)}{i+k+1-(j+k+1)}=-\binom{n-k-j}{i-j}$. Thus, for the entries of $Q_{n-k}$ exactly the same conditions apply as for those of $B_{n-k}$ and hence $Q_{n-k}=B_{n-k}$. From \eqref{eba17405} with the two results just obtained, we conclude that $\det(M_{k+1})=2^kb_{n-k}$. Substituting this into \eqref{em17405} yields \eqref{ev17405}.
\end{proof}

\begin{theorem}\label{Dow1}
Let $n\in\NN$ and consider the square matrix $A_n$ of size $n$ given by
\[
A_n(i,j) =
\begin{cases}
1, & \textnormal{if } i<j-1,\\
-1, & \textnormal{if } i=j-1,\\
\binom{n-j}{i-j}, & \textnormal{if } i\geq j,
\end{cases}
\qquad i,j\in[n].
\] 
Then, for every $n\in\NN_0$,
\[
\det(A_{n+1})=D_n.
\]\end{theorem} 

\begin{proof} 
Set $d_n = \det(A_{n+1})$. Since $A_1=(1)$, we have $\det(A_1)=1=D_0$, and the assertion holds in this case. Thus, we assume that $n\geq 1$ and therefore analyze $A_n$ for $n\geq 2$. For each $j=n,n-1,\ldots,2$, subtract from the $j$th column of $A_n$ its $j-1$th column. Let $\widetilde A_n$ denote the resulting matrix. It is straightforward to verify that 
\[ 
\widetilde A_n(i,j) = \begin{cases} 
\binom{n-1}{i-1},& \textnormal{if } j=1,\\ 
0,& \textnormal{if } i<j-2,\\ 
2,& \textnormal{if } i=j-2,\\ 
-2,& \textnormal{if } i=j-1,\\ 
-\binom{n-j}{i-j+1},& \textnormal{if } i\geq j\geq 2, 
\end{cases}\qquad i,j\in[n].
\] In particular, the last row of $\widetilde A_n$ is $(1,0,\dots,0)$. Therefore, expanding $\det(\widetilde A_n)$ along the last row gives \[\det(A_n)=\det(\widetilde A_n)=(-1)^{n+1}\det(B_{n-1}),\] where $B_{n-1}$ is the square matrix of size $n-1$ obtained by deleting the last row and the first column of $\widetilde A_n$.  Set $b_0=1$ and, for $n\geq 1$, set  $b_n=\det(B_n)$. Thus, $d_n=(-1)^nb_n$. By Lemma \ref{l17405},
\begin{equation}\label{ef17405}
b_{n+1}=-b_n-\sum_{k=0}^n (-1)^k\binom{n}{k}2^k b_{n-k}.
\end{equation} Multiplying both sides of \eqref{ef17405} by $(-1)^{n+1}$ and using $d_n=(-1)^n b_n$, we obtain, \begin{equation}\label{es1} d_{n+1}=d_n+\sum_{k=0}^n \binom{n}{k}2^{n-k}d_k. \end{equation} Let $F(x)=\sum_{n=0}^\infty d_n \frac{x^n}{n!}$ be the exponential generating function of the sequence $(d_n)_{n\in\NN_0}$. 
Multiplying \eqref{es1} by $\frac{x^n}{n!}$ and summing over $n\in\NN_0$ yields \[ F'(x)=F(x)+F(x)\sum_{n=0}^\infty \frac{(2x)^n}{n!}=(1+e^{2x})F(x).
\] Since $F(0)=d_0=1$,
\[F(x)=\exp\left(\int_0^x (1+e^{2t})dt\right) =\exp\left(x+\frac{e^{2x}-1}{2}\right),\] exactly the exponential generating function of the Dowling numbers, stated in \eqref{doweq}. Thus, $D_n=d_n =\det(A_{n+1})$, for every $n\in\NN_0$, as asserted.
\end{proof} 

\subsection{\texorpdfstring{\seqnum{A071999}}{}}

In the following result we obtain a closed form formula for the determinant of a matrix which is almost a cross matrix.

\begin{theorem}
For $n\in\NN$ let $A_n$ be the square matrix of size $n$ defined by
\[
A_n(i,j)=
\begin{cases}
1,& \textnormal{if $i=j$},\\
i,& \textnormal{if $i+j=n$},\\
0,& \textnormal{otherwise},
\end{cases} \qquad i,j\in[n].
\]
Then
\[
\det(A_n)=\prod_{i=1}^{\left\lfloor\frac{n-1}{2}\right\rfloor}(1-i(n-i)).
\]
\end{theorem}

\begin{proof}
First notice that for $i=n$ the condition $i+j=n$ forces $j=0$, which is impossible since $1\le j\le n$. Thus, the $n$th row of $A_n$ has only one nonzero entry, namely, $A_n(n,n)=1$. Expanding the determinant of $A_n$ along the $n$th row yields $\det(A_n)=\det(B_{n-1})$,
where $B_{n-1}$ is the cross matrix of size $n-1$ given by
\[
B_{n-1}(i,j)=
\begin{cases}
1,& \textnormal{if $i=j$},\\
i,& \textnormal{if $i+j=n$},\\
0,& \textnormal{otherwise},
\end{cases} \qquad i,j\in[n-1].
\] In the notation of Lemma \ref{l171999}, for every $1\leq i\leq \lfloor \frac{n-1}{2}\rfloor$,
\[
x_{ii}=1,\qquad x_{(n-1)+1-i,(n-1)+1-i}=1,\qquad
x_{i,(n-1)+1-i}=i,\qquad x_{(n-1)+1-i,i}=n-i.
\]
By \eqref{e171999},
\begin{align*}
\det(B_{n-1})
&=\prod_{i=1}^{\left\lfloor\frac{n-1}{2}\right\rfloor}(x_{ii}x_{(n-1)+1-i,(n-1)+1-i}-x_{i,(n-1)+1-i}x_{(n-1)+1-i,i}
)\\
&=\prod_{i=1}^{\left\lfloor\frac{n-1}{2}\right\rfloor}(1-i(n-i)).\qedhere
\end{align*}
\end{proof}

\subsection{\texorpdfstring{\seqnum{A083392}}{}}

The statement of the following theorem  confirms a conjecture by Spezia stated in \seqnum{A083392}.

\begin{theorem}
For $n\in\NN$ let $T_n$ be the symmetric Toeplitz matrix whose first row is
\[
0,1,1,2,2,3,3,\dots.
\] 
Then
$\det(T_n)=(-1)^{n-1}\left\lfloor \frac{n^2}{4}\right\rfloor$.
\end{theorem}

\begin{proof}
By definition of $T_n$, we have
\[T_n(i,j)=\left\lceil \frac{|i-j|}{2}\right\rceil, \qquad i,j\in[n].\]
For $n=1,2$ the assertion is verified by direct calculations. Thus, assume that $n\geq 3$. For each $i=n,n-1,\dots,3$, subtract the $i-2$th row from the $i$th row of $T_n$. Then, for each $j=n,n-1,\dots,3$, subtract the $j-2$th column from the $j$th column. Let $B_n$ denote the resulting matrix. We claim that 
$B_n$ is the block matrix
\[
B_n=
\begin{pmatrix}
A&F\\
F^T&C
\end{pmatrix},
\]
where $A$ is of size $2\times 2$, $F$ is of size $2\times(n-2)$, and $C$ is of size $(n-2)\times(n-2)$, and are given by 
\[
A=\begin{pmatrix}0&1\\1&0\end{pmatrix},
F=
\begin{pmatrix}
1&1&1&\cdots&1\\
0&1&1&\cdots&1
\end{pmatrix},
C=
\begin{pmatrix}
-2&-1&0&\cdots&0\\
-1&-2&-1&\ddots&\vdots\\
0&-1&-2&\ddots&0\\
\vdots&\ddots&\ddots&\ddots&-1\\
0&\cdots&0&-1&-2
\end{pmatrix}.
\]
Indeed, let $3\leq i,j\leq n$. Then
\begin{align*}
B_n(i,j)
&=
T_n(i,j)-T_n(i,j-2)-T_n(i-2,j)+T_n(i-2,j-2)\\
&=\left\lceil \frac{|i-j|}{2}\right\rceil-\left\lceil \frac{|i-j+2|}{2}\right\rceil
-
\left\lceil \frac{|i-2-j|}{2}\right\rceil
+
\left\lceil \frac{|i-j|}{2}\right\rceil\\
&=
\begin{cases}
-2,& \textnormal{if $|i-j|=0$},\\
-1,& \textnormal{if $|i-j|=1$},\\
0,& \textnormal{if $|i-j|\geq 2$}.
\end{cases}
\end{align*}
Similarly, for every $3\leq j \leq n$ we have
\begin{align*}
B_n(1,j)&=T_n(1,j)-T_n(1,j-2)=\left\lceil \frac{j-1}{2}\right\rceil-
\left\lceil \frac{j-3}{2}\right\rceil=1,\\
B_n(2,j)&=T_n(2,j)-T_n(2,j-2)=\left\lceil \frac{j-2}{2}\right\rceil-\left\lceil \frac{j-4}{2}\right\rceil
=\begin{cases}
0,&\textnormal{if $j=3$},\\
1,&\textnormal{if $j\geq 4$}.
\end{cases}
\end{align*}
By symmetry, the same holds for the first two columns. Finally, the \(2\times 2\) upper-left block remains unchanged during the operations and is originally given by $A$. By the second part of Lemma \ref{Schur},
\[
\det(B_n)=\det(C)\det(A-FC^{-1}F^T).
\] By Lemma \ref{tri}, $\det(C)=
(-1)^{n-2}(n-1)$. Furthermore, it is straightforward to verify that the inverse of $C$ is given by
\[
C^{-1}(i,j)=
\begin{cases}
(-1)^{i+j+1}\dfrac{i(n-1-j)}{n-1}, & \textnormal{if $i\le j$},\\
(-1)^{i+j+1}\dfrac{j(n-1-i)}{n-1}, & \textnormal{if $i> j$},
\end{cases}\qquad i,j\in[n-2].
\]
By direct calculations we obtain
\[
FC^{-1}F^T=-\frac{1}{n-1}
\begin{pmatrix}
\left\lfloor \frac{(n-1)^2}{4} \right\rfloor & \left\lfloor \frac{(n-2)^2}{4} \right\rfloor\\[8pt]
\left\lfloor \frac{(n-2)^2}{4} \right\rfloor & \left\lfloor \frac{(n-2)n}{4} \right\rfloor
\end{pmatrix}.
\] Hence,
\[A-FC^{-1}F^{T}=\begin{pmatrix}\frac{1}{n-1}\left\lfloor \frac{(n-1)^{2}}{4}\right\rfloor  & 1+\frac{1}{n-1}\left\lfloor \frac{(n-2)^{2}}{4}\right\rfloor \\[8pt]
1+\frac{1}{n-1}\left\lfloor \frac{(n-2)^{2}}{4}\right\rfloor  & \frac{1}{n-1}\left\lfloor \frac{(n-2)n}{4}\right\rfloor 
\end{pmatrix}.\] Thus,
\begin{align*}
&\det(A-FC^{-1}F^T)=\\
&\frac{1}{(n-1)^{2}}\left\lfloor \frac{(n-1)^{2}}{4}\right\rfloor \left\lfloor \frac{(n-2)n}{4}\right\rfloor -\left(1+\frac{1}{n-1}\left\lfloor \frac{(n-2)^{2}}{4}\right\rfloor \right)\left(1+\frac{1}{n-1}\left\lfloor \frac{(n-2)^{2}}{4}\right\rfloor \right)\\
&=-\frac{1}{n-1}\left\lfloor \frac{n^2}{4}\right\rfloor.
\end{align*} It follows that 
\begin{align*}
\det(T_n)=\det(B_n)&=\det(C)\det(A-FC^{-1}F^T)=\\
&=(-1)^{n-2}(n-1)\left(-\frac{1}{n-1}\left\lfloor \frac{n^2}{4}\right\rfloor\right)\\
&=(-1)^{n-1}\left\lfloor \frac{n^2}{4}\right\rfloor.\qedhere
\end{align*}
\end{proof}

\subsection{\texorpdfstring{\seqnum{A085799}}{}}

In the following result we obtain a closed form formula for the determinant of a certain symmetric matrix. Subsequently, we confirm a conjecture by Lajos stated in \seqnum{A085799}.

\begin{theorem}\label{t1}
Let $n\in\NN$ and let $A_n$ be the square matrix of size $n$ defined by
$A_n(i,j)=|i^2-j^2|, i,j\in[n]$. Then, for $n\geq 2$,
\[
\det(A_n)=(-1)^{n-1}\frac{n+1}{2}\cdot \frac{(2n-1)!}{(n-2)!}.
\]
\end{theorem}

\begin{proof}
Let $U_n$ be the square matrix of size $n$ given by
\[
U_n(i,j)=
\begin{cases}
1,  & \text{if } i=j,\\
-1, & \text{if } i=j+1,\\
0,  & \text{otherwise},
\end{cases} \qquad i,j\in[n].
\] Set $B_n = U_n A_n U_n^T$.
Since $\det(U_n)=1$, we have $\det(B_n)=\det(A_n)$. It is straightforward to verify that
\begin{align*}
&B_n(i,j)\\
&=
\begin{cases}
A_n(1,1), & \textnormal{if $i=j=1$},\\
A_n(1,j)-A_n(1,j-1), & \textnormal{if $i=1$ and $j\ge 2$},\\
A_n(i,1)-A_n(i-1,1), & \textnormal{if $j=1$ and $i\ge 2$},\\
A_n(i,j)-A_n(i-1,j)-A_n(i,j-1)+A_n(i-1,j-1), & \textnormal{if $i,j\ge 2$},
\end{cases}\\
&=
\begin{cases}
0, & \textnormal{if $i=j=1$},\\
2j-1, & \textnormal{if $i=1$ and $j\geq 2$},\\
2i-1, & \textnormal{if $j=1$ and $i\geq 2$},\\
-2(2i-1), & \textnormal{if $i\geq 2$ and $j=i$},\\
0, & \textnormal{if $i\geq 2$ and $j\neq i$}.
\end{cases}
\end{align*} Thus, $B_n$ is the symmetric arrowhead matrix
\[
B_n=
\begin{pmatrix}
0      & 3      & 5      & 7      & \cdots & 2n-1\\
3      & -2\cdot 3 & 0      & 0      & \cdots & 0\\
5      & 0      & -2\cdot 5 & 0      & \cdots & 0\\
7      & 0      & 0      & -2\cdot 7 & \ddots & \vdots\\
\vdots & \vdots & \vdots & \ddots & \ddots & 0\\
2n-1   & 0      & \cdots & \cdots & 0 & -2(2n-1)
\end{pmatrix}.
\] Thus, $B_n$ may be written as a block matrix
\[
B_n=
\begin{pmatrix}
0 & z^T\\
z & D
\end{pmatrix},
\] where
$z=(3, 5, \ldots, 2n-1)^T$ and 
$D=-2\diag(z)$. Since $D$ is invertible, by the second part of Lemma \ref{Schur},
\begin{align*}
\det(B_n)&=-\det(D)z^TD^{-1}z\\
&=-\left(\prod_{k=2}^n(-2(2k-1))\right)\sum_{k=2}^n \frac{(2k-1)^2}{-2(2k-1)}\\
&=-((-2)^{n-1}(2n-1)!!)\left(-\frac12(n^2-1)\right)\\
&=(-1)^{n-1}\frac{n+1}{2}\cdot \frac{(2n-1)!}{(n-2)!}.\qedhere
\end{align*}
\end{proof}

\begin{corollary}
Let $n\in\NN$ and let $A_n$ be the matrix defined in Theorem \ref{t1}. Then the Maple expression 
\[
\frac12\sum_{j=0}^{n}\mathrm{count}(\mathrm{Permutation}(2n-1),\mathrm{size}=n+1)
\]
is equal to $|\det(A_n)|$.
\end{corollary}

\begin{proof}
In Maple, the expression $\mathrm{count}(\mathrm{Permutation}(m),\mathrm{size}=r)$ counts permutations of length $r$ drawn from $m$ symbols without repetition. Thus, for $n\geq 2$,
\[
\mathrm{count}(\mathrm{Permutation}(2n-1),\mathrm{size}=n+1)=\frac{(2n-1)!}{(n-2)!}.
\] Since the summand is independent of $j$, the sum merely multiplies the summand by $n+1$. The factor $\frac{1}{2}$ brings the expression exactly to $|\det(A_n)|$.
Notice that for $n=1$ there are no permutations of length 
$2$ drawn without repetition from $1$ symbol and Maple returns $0$, which is indeed $|\det(A_1)|$.
\end{proof}

\subsection{\texorpdfstring{\seqnum{A323254}}{}}

In the following result we obtain a closed form formula for the determinant of a certain Toeplitz matrix with which \seqnum{A323254} is concerned. 

\begin{theorem}
For $n\in\NN$ let $T_n$ be the Toeplitz matrix whose first row is
\[
(2n-1, n-1, n-2,\dots,1)
\]
and whose first column is
\[
(2n-1, 2n-2, 2n-3,\dots, n).
\] Then
\[
\det(T_n)= n(n+1)^{n-1}+\frac{n-1}{4}((n-1)^{n-1}+(n+1)^{n-1}).
\]
\end{theorem}

\begin{proof}
We have $T_1=(1)$ and $\det(T_1)=1$, agreeing with the asserted formula. Thus, assume that $n\geq 2$. First, we notice that 
\[
T_n(i,j)=
\begin{cases}
2n-1-(i-j), & \textnormal{if $i\ge j$},\\
n-(j-i), & \textnormal{if $i<j$},
\end{cases}\qquad i,j\in[n].
\] 
For each $i=n,n-1,\ldots,2$ subtract the $i-1$th row of $T_n$ from the $i$th row and let $N_n$ denote the resulting matrix. It is straightforward to verify that
\[
N_n(i,j)=
\begin{cases}
2n-1,& \textnormal{if $i=j=1$},\\
n-j+1,& \textnormal{if $i=1$ and $j\geq 2$},\\
-1, & \textnormal{if $i\geq 2$ and $j\le i-1$},\\
n,  & \textnormal{if $i\geq 2$ and $j=i$},\\
1,  & \textnormal{if $i\geq 2$ and $j\ge i+1$},
\end{cases}\qquad i,j\in[n].
\]
Now, for each $j=1,2,\ldots,n-1$, subtract the $j+1$th column of $N_n$ from the $j$th column and let $K_n$ denote the resulting matrix. It is straightforward to verify that
\[
K_n(i,j)=
\begin{cases}
n, & \textnormal{if } i=j=1,\\[2pt]
1, & \textnormal{if } i=1 \textnormal{ and } 2\le j\le n,\\[2pt]
-(n+1), & \textnormal{if } 2\le i\le n \textnormal{ and } j=i-1,\\[2pt]
n-1, & \textnormal{if } 2\le i\le n-1 \textnormal{ and } j=i,\\[2pt]
1, & \textnormal{if } 2\le i\le n-1 \textnormal{ and } j=n,\\[2pt]
n, & \textnormal{if } i=j=n,\\[2pt]
0, & \textnormal{otherwise},
\end{cases}\qquad i,j\in[n].
\]
Thus, $K_n$ may be written as a block matrix
\[
K_n=\begin{pmatrix}
B_n & \mathbf{1}\\
c^T & n
\end{pmatrix},
\]
where $\mathbf{1}$ is the all-ones vector of size $n-1$, $c\in\RR^{n-1}$ is given by $c^T=(0,\dots,0,-(n+1))$,
and $B_n$ is the square matrix of size $n-1$ given by \[
B_n(i,j)=
\begin{cases}
n, & \textnormal{if } i=1 \textnormal{ and } j=1,\\
1, & \textnormal{if } i=1 \textnormal{ and } 2\le j\le n-1,\\
-(n+1), & \textnormal{if } 2\le i\le n-1 \textnormal{ and } j=i-1,\\
n-1, & \textnormal{if } 2\le i\le n-1 \textnormal{ and } j=i,\\
0, & \textnormal{otherwise},
\end{cases}\qquad i,j\in[n-1].
\] Expanding $\det(B_n)$ along the first row yields
\[
\det(B_n)=\sum_{j=1}^{n-1} B_n(1,j) (C_n)_{1j},
\]
where the $(C_n)_{1j}$s are the corresponding cofactors. By induction on $j\in[n-1]$, it is not hard to see that 
$(C_n)_{1j}=(n-1)^{n-1-j}(n+1)^{j-1}$. Thus,
\begin{align*}
\det(B_n)&=n(n-1)^{n-2}+\sum_{j=1}^{n-2} (n-1)^{n-2-j}(n+1)^j\\
&=\frac{(n-1)^{n-1}+(n+1)^{n-1}}{2}.    
\end{align*}
In particular, $B_n$ is invertible and by the first part of Lemma \ref{Schur},
\[
\det(K_n)=\det(B_n)(n-c^T B_n^{-1}\mathbf{1}).
\]
Let $x=B_n^{-1}\mathbf{1}$. To calculate $x$, we solve the system $B_n x=\mathbf{1}$, which may be written as 
\begin{align}
n x_1+x_2+\cdots+x_{n-1}&=1,\label{ee254}\\
-(n+1) x_{i-1}+(n-1)x_i&=1,\qquad 2\le i\le n-1.\label{eg1}
\end{align}
From \eqref{eg1} we obtain the recurrence
\begin{equation}\label{eg2254}
x_i=\frac{1}{n-1}+\frac{n+1}{n-1}x_{i-1},
\qquad 2\le i\le n-1.
\end{equation}
Solving \eqref{eg2254} gives
\begin{equation}\label{eww2254}
x_i=\left(\frac{n+1}{n-1}\right)^{i-1}x_1+\frac12\left(\left(\frac{n+1}{n-1}\right)^{i-1}-1\right), \qquad 2\leq i\leq n-1.
\end{equation}
From \eqref{ee254} and \eqref{eww2254} we obtain
\begin{equation}\label{ew2254}
x_1=\frac{2n}{n-1+\dfrac{(n+1)^{n-1}}{(n-1)^{n-2}}}-\frac12.
\end{equation} Taking $i=n-1$ in \eqref{eww2254} and using \eqref{ew2254}, we have
\[
x_{n-1}=\frac{2n(n+1)^{n-2}}{(n-1)^{n-1}+(n+1)^{n-1}}-\frac12.
\]
We conclude that 
\begin{align*}
\det(T_n)&=\det(N_n)=\det(K_n)=\det(B_n)(n-c^T x)\\
&=\det(B_n)(n+(n+1)x_{n-1})\\
&= \frac{(n-1)^{n-1}+(n+1)^{n-1}}{2}\left(n+(n+1)\left(\frac{2n(n+1)^{n-2}}{(n-1)^{n-1}+(n+1)^{n-1}}-\frac12\right)\right)\\
&=n(n+1)^{n-1}+\frac{n-1}{4}((n-1)^{n-1}+(n+1)^{n-1}).\qedhere
\end{align*}
\end{proof}

\subsection{\texorpdfstring{\seqnum{A351154}}{}}

The statement of the following theorem  confirms a conjecture by Spezia stated in \seqnum{A351154}.

\begin{theorem}
For $n\in\NN$ let $A_n$ be the square matrix of size $n$ defined by \[A_n(i,j)=f(n,\min(i,j)) + |i-j|, \qquad i,j\in[n],\] where 
\[f(n,k)=n(k-1)-\frac{k(k-3)}{2},\qquad k\in[n].\] 
Then, for $n\geq 2$, 
$\det(A_n)=-(n-2)!$.
\end{theorem}

\begin{proof}
By direct calculations it is immediately verified that the assertion holds for $n=2$. Thus, assume that $n\geq 3$. By definition of $A_n$,
\[
A_n(i,j)=
\begin{cases}
f(n,i)+j-i,& \textnormal{if $i\leq j$},\\    
f(n,j)+i-j, &\textnormal{if $i> j$},
\end{cases}\qquad i,j\in[n].
\]
For every $2\le k\le n$ we have $f(n,k)-f(n,k-1)=n-k+2$. For each $i=n,n-1,\dots,2$, subtract the $i-1$th row of $A_n$ from the $i$th row and let $B_n$ denote the resulting matrix. It is straightforward to verify that
\[
B_n(i,j)=
\begin{cases}
j, &\textnormal{if $i=1$},\\
1, & \textnormal{if $i\geq 2$ and $j<i$},\\
n-i+1, & \textnormal{if $i\geq 2$ and $j\ge i$},
\end{cases} \qquad i,j\in[n].
\] For each $j=n,n-1,\dots,2$, subtract the $j-1$th column of $B_n$ from the $j$th column and let $C_n$ denote the resulting matrix. Then
\[
C_n(i,j)=
\begin{cases}
1, &\textnormal{if $i=1$ or $j=1$},\\
n-i, & \textnormal{if $i\geq 2$ and $j=i$},\\
0, & \textnormal{otherwise},
\end{cases} \qquad i,j\in[n].
\] Let $R_n$ be the matrix obtained from $C_n$ by deleting its first column and last row. Then
\[
R_n(i,j)=
\begin{cases}
1, & \textnormal{if } i=1,\\
n-i, & \textnormal{if } i\geq 2 \textnormal{ and } j=i-1,\\
0, & \textnormal{otherwise},
\end{cases}
\qquad i,j\in[n-1].
\] Expanding $\det(C_n)$ along the last row of $C_n$, which is $(1, 0, \ldots,0)$, gives
$\det(C_n)=(-1)^{n+1}\det(R_n)$. Let $S_n$ be the matrix obtained from $R_n$ by deleting its first row and last column. Expanding $\det(R_n)$ along its last column, which is $(1,0,\ldots,0)^T$, gives
$\det(R_n) = (-1)^n \det(S_n)$. We have $S_n=\diag(n-2,n-3,\ldots,1)$. Thus,
\[\det(A_n)=(-1)^{n+1}(-1)^n (n-2)! = -(n-2)!. \qedhere\]
\end{proof}

\subsection{\texorpdfstring{\seqnum{A359559}}{}}

\noindent Sequence \seqnum{A359559} corresponds to the determinant of a certain Hermitian Toeplitz matrix. Using LU decomposition we obtain a closed form formula for the determinant. This immediately settles the conjectures regarding recurrence and ordinary and exponential generating functions. In this section $i=\sqrt{-1}$.

\begin{theorem}\label{t59}
Set $T_0 = (1)$ and, for $n\in\NN$, let $T_n$ be the Hermitian Toeplitz matrix of size $n$  whose first row is
$(1,2i,3i,\ldots,ni)$, i.e., 
\[
T_n(r,c)=
\begin{cases}
1, & \textnormal{if $r=c$},\\
i(c-r+1), & \textnormal{if $c>r$},\\
-i(r-c+1), & \textnormal{if $r>c$},
\end{cases}\qquad r,c\in[n].
\]
Set $\Delta_0=1$ and, for $k\in[n]$, let

\begin{align*}
\Delta_k&=
\frac{k^2+(4+i)k+4+4i}{8}(1+i)^k
+
\frac{k^2+(4-i)k+4-4i}{8}(1-i)^k,\\
\mu_k&=
\begin{cases}
0, & \textnormal{if $k=1$},\\
\left(-\dfrac{1+i}{8}k^2+\dfrac{1-i}{8}k+\dfrac{1+5i}{8}\right)(1+i)^k
+\dfrac{3-i}{8}(1-i)^k, & \textnormal{if $k\ge 2$},
\end{cases}\\
\gamma_k&=
\begin{cases}
i, & \textnormal{if $k=1$},\\
\left(\dfrac{1+i}{4}k+\dfrac{1+2i}{4}\right)(1+i)^k
+\dfrac{1}{4}(1-i)^k, & \textnormal{if $k\ge 2$}.
\end{cases}
\end{align*}
Then $T_n$ has an $LU$ decomposition $T_n=L_nU_n$ where, for $r,c\in[n]$,
\begin{align*} 
U_n(r,c)&=
\begin{cases}
0, & \textnormal{if $r>c$},\\
\dfrac{\Delta_r}{\Delta_{r-1}}, & \textnormal{if $r=c$},\\
\dfrac{\mu_r+c\gamma_r}{\Delta_{r-1}}, & \textnormal{if $c\ge r+1$},
\end{cases}\\
L_n(r,c)&=
\begin{cases}
0, & \textnormal{if $r<c$},\\
1, & \textnormal{if $r=c$},\\
\dfrac{\overline{\mu_c+r\gamma_c}}{\Delta_c}, & \textnormal{if $r\ge c+1$}.
\end{cases}
\end{align*}
\end{theorem}

\begin{proof}
It is straightforward but tedious to verify the decomposition by direct calculations. For additional confirmation, the following Python program verifies the identity $T_n=L_nU_n$ in exact arithmetic for $1\leq n\leq 10$.

\begin{verbatim}
import sympy as sp

I = sp.I

def Delta(k):
    if k == 0:
        return sp.Integer(1)

    return sp.simplify(
        sp.Rational(1, 8) * (k**2 + (4 + I) * k + 4 + 4 * I) * (1 + I)**k
        + sp.Rational(1, 8) * (k**2 + (4 - I) * k + 4 - 4 * I) * (1 - I)**k
    )

def mu(k):
    if k == 1:
        return sp.Integer(0)

    return sp.simplify(
        (
            -sp.Rational(1, 8) * (1 + I) * k**2
            + sp.Rational(1, 8) * (1 - I) * k
            + sp.Rational(1, 8) * (1 + 5 * I)
        ) * (1 + I)**k
        + sp.Rational(1, 8) * (3 - I) * (1 - I)**k
    )

def gamma(k):
    if k == 1:
        return I

    return sp.simplify(
        (
            sp.Rational(1, 4) * (1 + I) * k
            + sp.Rational(1, 4) * (1 + 2 * I)
        ) * (1 + I)**k
        + sp.Rational(1, 4) * (1 - I)**k
    )

def T_matrix(n):
    T = sp.zeros(n, n)

    for r in range(1, n + 1):
        for c in range(1, n + 1):
            if r == c:
                T[r - 1, c - 1] = 1
            elif c > r:
                T[r - 1, c - 1] = I * (c - r + 1)
            else:
                T[r - 1, c - 1] = -I * (r - c + 1)

    return T


def U_matrix(n):
    U = sp.zeros(n, n)

    for r in range(1, n + 1):
        for c in range(1, n + 1):
            if r > c:
                U[r - 1, c - 1] = 0
            elif r == c:
                U[r - 1, c - 1] = Delta(r) / Delta(r - 1)
            else:
                U[r - 1, c - 1] = (mu(r) + c * gamma(r)) / Delta(r - 1)

    return U

def L_matrix(n):
    L = sp.zeros(n, n)

    for r in range(1, n + 1):
        for c in range(1, n + 1):
            if r < c:
                L[r - 1, c - 1] = 0
            elif r == c:
                L[r - 1, c - 1] = 1
            else:
                L[r - 1, c - 1] = sp.conjugate(mu(c) + r * gamma(c)) / Delta(c)

    return L

def verify_LU(n):
    T = T_matrix(n)
    L = L_matrix(n)
    U = U_matrix(n)

    difference = (L * U - T).applyfunc(sp.simplify)

    if difference == sp.zeros(n, n):
        print(f"n={n}: OK")
        return True

    print(f"n={n}: FAILED")
    print("L*U - T =")
    sp.pprint(difference)
    return False

for n in range(1, 11):
    verify_LU(n)
\end{verbatim}

\end{proof}

\begin{corollary}\label{c1}
Let $n\in\NN_0$ and consider the matrix $T_n$ from Theorem \ref{t59}. Then, 
\[\det(T_n)=
\frac{n^2+(4+i)n+4+4i}{8}(1+i)^n
+
\frac{n^2+(4-i)n+4-4i}{8}(1-i)^n.
\]
\end{corollary}

\begin{proof}
Direct calculations show that the assertion holds for $n=0$. Thus, assume that $n\geq 1$. We have $\det(T_n)=\det(L_n)\det(U_n)$. The matrix $L_n$ is a lower triangular matrix with $1$s on the main diagonal. Thus, $\det(L_n)=1$. The matrix $U_n$ is an upper triangular matrix with $\frac{\Delta_r}{\Delta_{r-1}}, 1\leq r\leq n$ on the main diagonal. Thus,
\[\det(T_n)=\det(U_n)=\prod_{r=1}^n\frac{\Delta_r}{\Delta_{r-1}}=\frac{\Delta_n}{\Delta_0}=\Delta_n.\qedhere\] 
\end{proof}

\begin{corollary}
For $n\in\NN_0$, let $t_n=\det(T_n)$, where $T_n$ is the matrix from Theorem \ref{t59}. Then $t_n$ satisfies the recurrence
\[
t_{n+6}=6t_{n+5}-18t_{n+4}+32t_{n+3}-36t_{n+2}+24t_{n+1}-8t_n.
\] Furthermore, the ordinary generating function for the $t_n$s is given by
\[\frac{1 - 5x + 9x^2 - 12x^3 + 10x^4 - 4x^5}{(1 - 2x + 2x^2)^3}.
\] Finally, the corresponding exponential generating function is given by
\[\frac{e^{x}}{2}(2(x+1)\cos x-(x^{2}+3x+2)\sin x).\] 
\end{corollary}

\begin{proof}
By Corollary \ref{c1}, $t_n = \alpha_1(n)(1+i)^n + \alpha_2(n)(1-i)^n$, where $\alpha_1(n)$ and $\alpha_2(n)$ are two polynomials in $n$ of degree $2$. Thus, $t_n$ is annihilated by $(E-(1+i))^3(E-(1-i))^3$, where $E$ is the shift operator $Et_n=t_{n+1}$. Expanding
\[
(E-(1+i))^3(E-(1-i))^3= E^6 - 6E^5 + 18E^4 - 32E^3 + 36E^2 - 24E + 8,
\] we obtain the asserted recurrence.
\end{proof}

\subsection{\texorpdfstring{\seqnum{A355175}}{}}

For $m,n\in\NN$ let $A_{n}^{(m)}$ be the square matrix of size $n$ given by 
\begin{equation}\label{nm1}
A_{n}^{(m)}(i,j)=(i-j)^m+\delta_{ij},\qquad i,j\in[n]. 
\end{equation} In \cite{WW} Wang and Sun evaluated several Toeplitz-type determinants, one of which is the determinant of $A_{n}^{(1)}$. In \seqnum{A355175} and \seqnum{A355326} Sun conjectured closed-form formulas for $\det(A_{n}^{(2)})$ and $\det(A_{n}^{(3)})$, respectively. Sun also made the general conjecture that $\det(A_{n}^{(m)})$ has the form $1 + n^2(n^2-1)p_m(n)$, where $p_m(n)$ is a rational polynomial in $n$ with degree $(m+1)^2-4$. Following an approach suggested by a referee in \cite{WW}, namely, by using Sylvester's determinant theorem, we resolve all of Sun's conjectures mentioned above. 

For a polynomial $p(x)$ let $[x^n]p(x)$ denote the coefficient of $x^n$ in $p(x)$. For $r\in\NN_0$ and $n\in\NN$ set $s_r(n)=\sum_{k=1}^n k^r$. For $r\in\NN_0$ let $B_r$ denote the $r$th Bernoulli number and let $B_r(x)$ denote the $r$th Bernoulli polynomial. Recall that 
\[B_r(x) = \sum_{k=0}^r\binom{r}{k}B_kx^{r-k}.\] It is well-known  (e.g., \cite[(7.79)]{G}) that for every $r\in\NN_0$, 
\begin{equation}\label{eq1175}
0^r+s_r(n)=\frac{B_{r+1}(n+1)-B_{r+1}(0)}{r+1}.
\end{equation} In particular, $s_r(n)$ is a polynomial in $n$ of degree $r+1$. Indeed, since $\deg(B_{r+1}(x))=r+1$, the polynomial $B_{r+1}(n+1)$ in $n$ is of degree $r+1$. Subtracting the constant $B_{r+1}(0)$ and dividing by $r+1$ do not change the degree.

\begin{theorem}\label{ttT}
Let $m,n\in\NN$ and let $S_n^{(m)}$ be the square matrix of size $m+1$ given by
\[S_n^{(m)}(i,j)=\delta_{ij}+(-1)^i\binom{m}{i}s_{i+m-j}(n),\qquad 0\le i,j\le m.
\] Then
\begin{equation}\label{e1175}
\det(A_{n}^{(m)})=\det(S_n^{(m)}).
\end{equation}
\end{theorem}

\begin{proof}
We have $A_{n}^{(m)}=I_n+M^{(m)}_n$, where $M^{(m)}_n$ is the square matrix of size $n$ given by \[M^{(m)}_n(i,j)=(i-j)^m,\qquad i,j\in[n].\] For $r\in\NN_0$ set $v_r=(1^r,2^r,\ldots,n^r)^T$. By the binomial theorem,
\[
(i-j)^m=\sum_{k=0}^m (-1)^k\binom{m}{k} i^{m-k} j^k.
\] Thus,
\[
M^{(m)}_n=\sum_{k=0}^m (-1)^k\binom{m}{k} v_{m-k} v_k^T.
\] Let $U$ and $V$ be the two matrices of size $n\times(m+1)$ whose $k$th column is given by
$v_{m-k}$ and $(-1)^k\binom{m}{k}v_k$, respectively, $k=0,1,\dots,m$. Thus, $M^{(m)}_n=UV^T$ and therefore $A_{n}^{(m)}=I_n+UV^T$. By \eqref{SYL},  \[\det(A_{n}^{(m)})=\det(I_{m+1}+V^TU).
\]
Now, for every $0\le i,j\le m$ we have
\[
(V^TU)(i,j)=V_i^TU_j
=\left((-1)^i\binom{m}{i} v_i\right)^T v_{m-j}
=(-1)^i\binom{m}{i}s_{i+m-j}(n).
\] Thus, $S_{n}^{(m)} = I_{m+1}+V^TU$ and we are done.
\end{proof}

The following result confirms the conjectured closed-form formulas in \seqnum{A079034}, \seqnum{A355175}, and \seqnum{A355326}, respectively.  

\begin{corollary}
Let $n\in\NN$. Then 
\begin{align*}
\det(A_{n}^{(1)})&=1+n^2(n^2-1)\frac{1}{12},\\
\det(A_{n}^{(2)})
&=1+n^2(n^2-1)\frac{n^5-5n^3-36n^2+4n+54}{1080},\\
\det(A_{n}^{(3)})
&=1+n^2(n^2-1)\frac{n^{12}-19n^{10}+123n^8-337n^6+12376n^4-44144n^2+40000}{672000}. 
\end{align*}
\end{corollary}

\begin{proof}
By \eqref{e1175}, 
\begin{align*}
\det(A_{n}^{(1)})&=
\det\begin{pmatrix}
1+s_1(n) & s_0(n)\\
-s_2(n)  & 1-s_1(n)
\end{pmatrix}, \\
\det(A_{n}^{(2)})&=
\det\begin{pmatrix}
1+s_2(n) & s_1(n) & s_0(n)\\
-2s_3(n) & 1-2s_2(n) & -2s_1(n)\\
s_4(n) & s_3(n) & 1+s_2(n)
\end{pmatrix},\\
\det(A_{n}^{(3)})&=
\det\begin{pmatrix}
1+s_3(n) & s_2(n) & s_1(n) & s_0(n)\\
-3s_4(n) & 1-3s_3(n) & -3s_2(n) & -3s_1(n)\\
3s_5(n) & 3s_4(n) & 1+3s_3(n) & 3s_2(n)\\
-s_6(n) & -s_5(n) & -s_4(n) & 1-s_3(n)
\end{pmatrix}.
\end{align*}
Evaluating the determinant in each of the cases, with Faulhaber's formulas and algebraic simplification, the closed-form formulas follow. For additional confirmation, the following Python program verifies the three formulas.

\begin{verbatim}
import sympy as sp

n = sp.symbols("n", integer=True, positive=True)

def s_symbol(r):
    return sp.Symbol(f"s_{r}(n)")

def s_power_poly(r):
    k = sp.symbols("k", integer=True, positive=True)
    return sp.summation(k**r, (k, 1, n))

def S_matrix_with_s_symbols(m):
    S = sp.zeros(m + 1, m + 1)

    for i in range(m + 1):
        for j in range(m + 1):
            delta = 1 if i == j else 0
            S[i, j] = delta + (-1)**i * sp.binomial(m, i) * s_symbol(i + m - j)

    return S

def S_matrix_exact_symbolic(m):
    S = sp.zeros(m + 1, m + 1)

    for i in range(m + 1):
        for j in range(m + 1):
            delta = 1 if i == j else 0
            S[i, j] = delta + (-1)**i * sp.binomial(m, i) * s_power_poly(i + m - j)

    return S

def closed_form(m):
    if m == 1:
        return 1 + n**2 * (n**2 - 1) / 12

    if m == 2:
        return 1 + n**2 * (n**2 - 1) * (
            n**5 - 5 * n**3 - 36 * n**2 + 4 * n + 54 ) / 1080

    if m == 3:
        return 1 + n**2 * (n**2 - 1) * (
            n**12 - 19 * n**10 + 123 * n**8 - 337 * n**6
            + 12376 * n**4 - 44144 * n**2 + 40000 ) / 672000

def verify_symbolically(m):
    print(f"m = {m}")
    print("S_n^(m):")
    print(S_matrix_with_s_symbols(m))

    det_S = sp.factor(S_matrix_exact_symbolic(m).det())
    formula = sp.factor(closed_form(m))
    difference = sp.factor(det_S - formula)

    print("symbolic determinant:")
    print(det_S)

    if difference == 0:
        print("agrees with closed form: yes")
    else:
        print("agrees with closed form: no")        

    print()

for m in [1, 2, 3]:
    verify_symbolically(m)    
\end{verbatim}
\end{proof}

\begin{theorem}
Let $m,n\in\NN$ and consider the matrix $S_n^{(m)}$ defined in Theorem \ref{ttT}. Set $D_m(n)=\det(S_n^{(m)})$. Then there exists a polynomial
$p_m(n)\in\QQ[n]$ of degree $(m+1)^2-4$ such that
\[
D_m(n)=1+n^2(n^2-1)p_m(n).
\]
\end{theorem}

\begin{proof}
Let $r\in\NN_0$. By Faulhaber's formulas, $s_r(n)\in\QQ[n]$. Now, by \eqref{e1175}, $D_m(1)=\det(A_1^{(m)})=1$. Thus, $D_m(n)-1$ is divisible by $n-1$.

By \eqref{eq1175}, $s_r(-1)=0$ for every $r\geq 1$ and therefore $s_{i+m-j}(-1)=0$ for every $0\leq i,j\leq m$, except when $i=0$ and $j=m$, and, in this case, $s_0(-1)=-1$. Consequently, $S_{-1}^{(m)}=I_{m+1}+E$, where $E$ is the square matrix of size $m+1$ given by 
\[
E(i,j)=
\begin{cases}
-1, &\textnormal{if } (i,j)=(0,m),\\
0, &\textnormal{otherwise},
\end{cases} \qquad 0\le i,j\leq m. \]
Hence, $D_m(-1)=\det(S_{-1}^{(m)})=1$ and therefore $D_m(n)-1$ is divisible by $n+1$. This concludes the proof that $n^2-1$ divides $D_m(n)-1$. 

We now wish to prove that $D_m(n)-1$ is divisible by $n^2$. First, notice that since $s_r(0)=0$, we have $S_0^{(m)}=I_{m+1}$. Thus, $D_m(0)=\det(S_0^{(m)})=1$. This shows that $n$ divides $D_m(n)-1$. By Lemma \ref{Jac},
\[
\frac{d}{dn}\det(S_n^{(m)})=\det(S_n^{(m)})\tr\left((S_n^{(m)})^{-1}\frac{d}{dn}S_n^{(m)}\right).
\]
Now,
\[\frac{d}{dn}S_n^{(m)}(i,j)=(-1)^i\binom{m}{i}\frac{d}{dn}s_{i+m-j}(n), \qquad 0\leq i,j\leq m.\] 
Thus,
\begin{align*}
\left(\frac{d}{dn}\det(S_n^{(m)})\right)(n=0)&=\det(I_{m+1})\tr\left(I_{m+1}^{-1}\left(\frac{d}{dn}S_n^{(m)}\right)(n=0)\right)\\
&=\tr\left(\left(\frac{d}{dn}S_n^{(m)}\right)(n=0)\right)\\
&=\left(\frac{d}{dn}s_m(n)\right)(n=0)\sum_{i=0}^{m}(-1)^i\binom{m}{i}\\
&=\left(\frac{d}{dn}s_m(n)\right)(n=0)(1-1)^m\\
&=0.
\end{align*} This concludes the proof that $n^2$ divides $D_m(n)-1$.

We shall now prove that $\deg(p_m(n)) = (m+1)^2-4$. By \eqref{eq1175}, $\deg(s_r(n))=r+1$. Let $\deg(S_n^{(m)})$ denote the matrix whose entries are the degrees in $n$ of the entries of $S_n^{(m)}$, which are polynomials in $n$. Then
\[
\deg(S_n^{(m)})(i,j)=i+m-j+1,\qquad 0\le i,j\le m.\] Let $\pi$ be a permutation of $\{0,1,\ldots,m\}$. We have
\[
\sum_{i=0}^m (i+m-\pi(i)+1)
=\sum_{i=0}^m (i+m+1)-\sum_{i=0}^m \pi(i)=(m+1)^2.
\] It follows that $\deg (D_m(n))\le (m+1)^2$. On the other hand, we have \[[n^{i+m-j+1}](S_n^{(m)})(i,j) = (-1)^i\binom{m}{i}\frac{1}{i+m-j+1}, \qquad 0\leq i,j\leq m.\] Thus, $[n^{(m+1)^2}]D_m(n)=\det(Q)$, where $Q$ is the square matrix of size $m+1$ given by
\[
Q(i,j) = (-1)^i\binom{m}{i}\frac{1}{i+m-j+1}, \qquad {0\le i,j\le m}.\]
The matrix $Q$ is a Cauchy matrix up to nonzero row factors. By Lemma \ref{Cauchy}, $\det(Q)\neq 0$. It follows that $\deg(D_m(n))=(m+1)^2$. Thus, factoring out the term $n^2(n^2-1)$, which has degree $4$, from $D_m(n)-1$, the remaining polynomial $p_m(n)$ has degree $(m+1)^2-4$.
\end{proof}

\subsection{\texorpdfstring{\seqnum{A250742}}{}}

The statement of the following theorem  confirms the conjectures by Hardin stated in \seqnum{A250742}.

\begin{theorem}
Let $n,k\in\NN_0$ and let $T(n,k)$ be the number of binary matrices $X$ of size $(n+1)\times (k+1)$ satisfying the following two conditions:
\begin{enumerate}
\item For every $j\in[k]$ the sequence $(X(i,j)-X(i,j-1))_{i=0}^n$ is nondecreasing in $i$.
\item For every $i\in[n]$ the sequence $(X(i,j)-X(i-1,j))_{j=0}^k$ is nonincreasing in $j$.
\end{enumerate}
Then
$T(n,k)=2^{n+1}+2^{k+1}-2$.
\end{theorem}

\begin{proof}
Let $X$ be a matrix satisfying the two conditions and let $i\in[n]$ and $j\in[k]$. Then
\begin{align}
X(i,j)-X(i,j-1)& \geq X(i-1,j)-X(i-1,j-1),\label{e1742}\\
X(i,j)-X(i-1,j)& \leq X(i,j-1)-X(i-1,j-1).\nonumber
\end{align} The second inequality is equivalent to 
\begin{equation}\label{e2742}
X(i,j)- X(i,j-1)\leq X(i-1,j)-X(i-1,j-1).
\end{equation}
From \eqref{e1742} and \eqref{e2742} it follows that
\begin{equation}\label{e3742}
X(i,j)-X(i,j-1)=X(i-1,j)-X(i-1,j-1).
\end{equation}
Let $j'\in[k]$. Iterating \eqref{e3742} over $i=1,2,\ldots,n$ we conclude that 
\begin{equation}\label{e5}
X(i,j')-X(i,j'-1)=X(0,j')-X(0,j'-1), \qquad 0\leq i\leq n.
\end{equation}
Now let $j\in[k]$. Summing \eqref{e5} over $j'\in[j]$ yields
\[
X(i,j)-X(i,0)=X(0,j)-X(0,0).
\]
Equivalently, 
\begin{equation}\label{ez1742}
X(i,j)=X(i,0)+X(0,j)-X(0,0).    
\end{equation}
It follows that every matrix satisfying the two conditions is already determined by its first row and first column. 

Conversely, consider a binary matrix $X$ of size $(n+1)\times(k+1)$ such that \eqref{ez1742} holds for every $0\leq i\leq n$ and $0\leq j\leq k$. Then 
\begin{align*}
&X(i,j)-X(i,j-1)\\
&=X(i,0)+X(0,j)-X(0,0)-(X(i,0)+X(0,j-1)-X(0,0))\\
&=X(0,j)-X(0,j-1).    
\end{align*} Thus, \eqref{e3742} holds and therefore also the two conditions. This concludes the proof that a matrix $X$ satisfies the two conditions if and only if it satisfies the condition \eqref{ez1742}. 

To count the number of binary matrices satisfying \eqref{ez1742}, it suffices to count the number of binary 
vectors $(u_0,\ldots,u_n), (v_0,\ldots,v_k)$, such that $u_0=v_0$ and $u_i+v_j-u_0\in\{0,1\}$ for every $0\le i\le n$ and $0\le j\le k$. We distinguish between the two possibilities for $u_0$.
\begin{enumerate}
\item $u_0=0$. Let $i\in[n],j\in[k]$. Then \[u_i+v_j\in\{0,1\}\iff (u_i,v_j)\neq (1,1).\] Thus, either
$v_1=\cdots=v_k=0$ and $u_1,\dots,u_n$ are arbitrary, yielding $2^n$ possibilities, or $u_1=\cdots=u_n=0$ and $v_1,\dots,v_k$ are arbitrary, yielding $2^k$ possibilities. The intersection of two scenarios consists of the unique possibility $u_1=\cdots=u_n=v_1=\cdots=v_k=0$, yielding a total of $2^n+2^k-1$ possibilities.

\item $u_0=1$. Let $i\in[n], j\in[k]$. Then \[u_i+v_j-1\in\{0,1\}\iff u_i+v_j\in\{1,2\}\iff (u_i,v_j)\neq(0,0).\] Thus, either $v_1=\cdots=v_k=1$ and $u_1,\dots,u_n$ are arbitrary, yielding $2^n$ possibilities, or $u_1=\cdots=u_n=1$ and $v_1,\dots,v_k$ are arbitrary, yielding $2^k$ possibilities. The intersection of two scenarios consists of the unique possibility $u_1=\cdots=u_n=v_1=\cdots=v_k=1$, yielding a total of $2^n+2^k-1$ possibilities.
\end{enumerate}
We conclude that
$T(n,k)=2(2^n+2^k-1)=2^{n+1}+2^{k+1}-2$, as asserted.
\end{proof}

\subsection{\texorpdfstring{\seqnum{A094384}}{}}

Let $m,n,p,q\in\NN$. Let $A$ be a matrix of size $m\times n$ and let $B$ be a matrix of size $p\times q$. Recall (e.g., \cite[Definition 4.2.1]{HJ}) that the Kronecker product of $A$ and $B$, denoted by $A\otimes B$, is defined to be the matrix of size $(mp)\times(nq)$ given by \[
A\otimes B
=
\begin{pmatrix}
A(1,1)B & \cdots & A(1,n)B\\
\vdots & \ddots & \vdots\\
A(m,1)B & \cdots & A(m,n)B
\end{pmatrix}.
\] For $k\in\NN$ the Kronecker $k$th power of a matrix $A$, denoted by $A^{\otimes k}$, is defined inductively by
\[
A^{\otimes k}=\begin{cases}
A, & \textnormal{if } k =1,\\
A\otimes A^{\otimes (k-1)},&\textnormal{if } k \geq 2.
\end{cases}\]

It is well known (e.g., \cite[Lemma 4.2.10]{HJ}) that for every four matrices $A,B,C$, and $D$, whenever the products are defined, we have
\begin{equation}\label{prod}
(AB)\otimes(CD)=(A\otimes C)(B\otimes D).
\end{equation} For a matrix $X$ and $n\in\NN$, let $X^{[n]}$ denote the upper-left square submatrix of size $n$ of $X$.
For $i\in\NN_0$ let $s_2(i)$ denote the number of $1$'s in the binary representation of $i$.

\begin{theorem}\label{t1384}
Let $M(0)=(1)$ and, for $m\in\NN$ let $M(m)$ be the square matrix of size $2^m$ defined by
\[
M(m)=\begin{pmatrix}M(m-1)&-M(m-1)\\-M(m-1)&-M(m-1)\end{pmatrix}.
\] For $m,n\in\NN$ such that $2^m\geq n$ let $M_n= M(m)^{[n]}$. Then 
\[
\det(M_n)=(-2)^{\sum_{i=1}^{n-1} s_2(i)}.
\]
\end{theorem}

\begin{proof}
First, we notice that for every $m\in\NN$, we have $M(m)=K\otimes M(m-1)$, where 
$K=\begin{pmatrix}1&-1\\-1&-1\end{pmatrix}$. By induction, $M(m)=K^{\otimes m}$. The matrix $K$ admits an $LU$-factorization $K=LU$ where 
$L=\begin{pmatrix}1&0\\-1&1\end{pmatrix}$ and $U=
\begin{pmatrix}1&-1\\0&-2\end{pmatrix}$. Using \eqref{prod} we inductively obtain
$K^{\otimes m}=L^{\otimes m}U^{\otimes m}$, for every $m\in\NN$. Clearly, $L^{\otimes m}$ is a lower triangular matrix with diagonal entries all equal to $1$ and $U^{\otimes m}$ is an upper triangular matrix. We claim that for every $m,n\in\NN$ with $2^m\ge n$,
\begin{equation}\label{e9384}
(K^{\otimes m})^{[n]} = (L^{\otimes m})^{[n]}(U^{\otimes m})^{[n]}.
\end{equation}
Indeed, let $i,j\in[n]$. Then 
\[
\left((L^{\otimes m})^{[n]}(U^{\otimes m})^{[n]}\right)(i,j)
=\sum_{t=1}^n L^{\otimes m}(i,t)U^{\otimes m}(t,j).
\]
Let $t\in[n]$. Since $L^{\otimes m}$ is lower triangular, if $t>i$ then $L^{\otimes m}(i,t)=0$. Since $U^{\otimes m}$ is upper triangular, if $t>j$ then $U^{\otimes m}(t,j)=0$. Thus, if $t>\min(i,j)$ then $L^{\otimes m}(i,t)U^{\otimes m}(t,j)=0$. Since $\min(i,j)\leq n$, we have
\[
\sum_{t=1}^{n} L^{\otimes m}(i,t)U^{\otimes m}(t,j)=\sum_{t=1}^{2^m} L^{\otimes m}(i,t)U^{\otimes m}(t,j) = L^{\otimes m}U^{\otimes m}(i,j) = K^{\otimes m}(i,j),
\]
proving \eqref{e9384}. It follows that
\begin{equation}\label{e23384}
\det(M_n)=\det((K^{\otimes m})^{[n]})
=\det((L^{\otimes m})^{[n]})\det((U^{\otimes m})^{[n]})
=\prod_{i=1}^{n}U^{\otimes m}(i,i).
\end{equation} For $m\in\NN$, the matrix $U^{\otimes m}$ is a square matrix of size $2^m$. Thus, indexing the rows and columns of $U^{\otimes m}$ by $0,1,\ldots,2^m-1$, a diagonal element of $U^{\otimes m}$ at position $(i,i)$ is naturally indexed by the binary representation $(b_0,\dots,b_{m-1})$ of $i$, i.e., $i=\sum_{\ell=0}^{m-1} b_\ell 2^\ell$, where $(b_0,\dots,b_{m-1})\in\{0,1\}^m$. We claim that for every $b=(b_0,\dots,b_{m-1})\in\{0,1\}^m$ we have
\begin{equation}\label{e10}
U^{\otimes m}(b,b)=\prod_{\ell=0}^{m-1} U(b_\ell, b_\ell).
\end{equation}
To see this, index the rows and columns of $U$ by $0$ and $1$. For $m=1$ let $b\in\{0,1\}^1$. Then
\[
U^{\otimes 1}(b,b)=U(b,b)=U_{b_0,b_0},
\] as it should. Now, assume that \eqref{e10} holds for some $m\in\NN$ and let $b=(b_0,\ldots,b_m)\in\{0,1\}^{m+1}$. We have
$U^{\otimes (m+1)}=U\otimes U^{\otimes m}$. Writing $b=(b_0,b')$ with $b_0\in\{0,1\}$ and
$b'=(b_1,\ldots,b_m)\in\{0,1\}^m$, and using the induction hypothesis, we have 
\[
U^{\otimes (m+1)}((b_0,b'),(b_0,b'))
=
U(b_0,b_0) U^{\otimes m}(b',b')=U(b_0,b_0)\prod_{\ell=1}^{m} U(b_\ell, b_\ell)=\prod_{\ell=0}^m U(b_\ell, b_\ell).
\] Since $U(0,0)=1$ and $U(1,1)=-2$, it follows from \eqref{e10} that for every $i\in\{0,1,\ldots,2^m-1\}$ with binary representation $(b_0,\ldots,b_{m-1})\in\{0,1\}^m$, we have
\begin{equation}\label{xx384}
U^{\otimes m}(i,i)=U^{\otimes m}(b,b)=\prod_{\ell=0}^{m-1} U(b_\ell, b_\ell) = (-2)^{s_2(i)}.
\end{equation} We now turn back to the standard indexing of the columns and rows of $U^{\otimes m}$. By \eqref{e23384} and \eqref{xx384}, 
\[
\det(M_n)=\prod_{i=1}^n U^{\otimes m}(i,i)
=\prod_{i=1}^n (-2)^{s_2(i-1)}
=(-2)^{\sum_{i=1}^n s_2(i-1)}=(-2)^{\sum_{i=1}^{n-1} s_2(i)}.\qedhere
\]
\end{proof}

\section*{Declaration of generative AI and AI-assisted technologies in the manuscript preparation process}

During the preparation of this work, the author used ChatGPT in order to improve the language, clarity, and presentation of parts of the manuscript. After using this tool, the author reviewed and edited the content as needed and takes full responsibility for the content of the published article.

\end{document}